\documentclass[10pt]{article}

\usepackage[margin=2.05cm]{geometry}
\usepackage{booktabs} 
\usepackage{amsfonts,amsthm,amssymb,amsmath}
\usepackage[utf8]{inputenc}
\RequirePackage[l2tabu, orthodox]{nag}

\usepackage[american]{babel}
\usepackage{xspace} 
\usepackage{dsfont}
\usepackage{microtype}
\usepackage{appendix}
\DisableLigatures[f]{encoding = *}

\usepackage{xcolor}
\definecolor{dred}{RGB}{220,0,0}
\newcommand{\puff}[1]{}

\usepackage{tabularx} 
\usepackage{caption}
\usepackage{subfig}  
\usepackage{placeins}
\usepackage{psfrag}
\usepackage{graphicx}
\usepackage{epstopdf}

\newcommand{\ie}{i.e., }
\newcommand{\eg}{e.g., }

\newcommand{\ed}{\,{\buildrel d \over =}\,}

\newcommand{\E}{\mathbb E}
\newcommand{\ind}[1]{\mathds{1}_{\{ #1\}}}

\newcommand{\N}{\mathbb N}
\newcommand{\R}{\mathbb R}

\newcommand{\e}{\varepsilon}

\renewcommand{\o}{\omega}

\newcommand{\G}{\Gamma}
\newcommand{\GC}{\Gamma(\nc)}

\newcommand{\cD}{\mathcal{D}_{\nc}}

\newcommand{\cS}{\mathcal{S}}

\newcommand{\cX}{\mathcal{X}}

\newcommand{\nc}{C}
\newcommand{\ac}{\mathcal{A}(C)}
\newcommand{\D}{\Delta}
\newcommand{\siv}{\Upsilon_i(\nc)}
\newcommand{\si}{\Upsilon(\nc)}

\newcommand{\ninf}{\nu \to \infty}

\newcommand{\tm}{t_{\mathrm{mix}}}
\newcommand{\limnu}{\lim_{\ninf}}

\theoremstyle{acmplain}
\newtheorem{theorem}{Theorem}[section]

\newtheorem{proposition}[theorem]{Proposition}
\newtheorem{lemma}[theorem]{Lemma}

\theoremstyle{acmdefinition}

\begin{document}
\title{Temporal starvation in multi-channel CSMA networks:\\an analytical framework}

\author{Alessandro Zocca\footnote{California Institute of Technology, Pasadena, US. Email: \texttt{azocca@caltech.edu}}}

%



\maketitle

\begin{abstract}
In this paper we consider a stochastic model for a frequency-agile CSMA protocol for wireless networks where multiple orthogonal frequency channels are available. Even when the possible interference on the different channels is described by different conflict graphs, we show that the network dynamics can be equivalently described as that of a single-channel CSMA algorithm on an appropriate virtual network.
Our focus is on the asymptotic regime in which the network nodes try to activate aggressively in order to achieve maximum throughput. Of particular interest is the scenario where the number of available channels is not sufficient for all nodes of the network to be simultaneously active and the well-studied temporal starvation issues of the single-channel CSMA dynamics persist.
For most networks we expect that a larger number of available channels should alleviate these temporal starvation issues. However, we prove that the aggregate throughput is a non-increasing function of the number of available channels. 
To investigate this trade-off that emerges between aggregate throughput and temporal starvation phenomena, we propose an analytical framework to study the transient dynamics of multi-channel CSMA networks by means of first hitting times.
Our analysis further reveals that the mixing time of the activity process does not always correctly characterize the temporal starvation in the multi-channel scenario and often leads to pessimistic performance estimates.
\end{abstract}

\section{Introduction}
\label{sec1}
The Carrier-Sense Multiple-Access (CSMA) algorithm is a popular distributed medium-access control mechanism and indeed various incarnations of which are currently implemented in IEEE 802.11 WiFi networks. CSMA is a random-access algorithm, in the sense that it relies on randomness to both avoid simultaneous transmissions and share the medium in the most efficient way. Being intrisincally a randomized scheme, many stochastic models have been developed in the literature to study its performance in terms of throughput, stability, delay, and spatial fairness, see \eg \cite{Yun2012} and references therein for a detailed overview.

It is well-known that the delay performance of CSMA algorithms can be rather poor and much worse than for other mechanisms, such as MaxWeight. One of the root causes for these poor delay performances has been identified in the \textit{temporal starvation phenomenon}: even in scenarios where all nodes have a good long-term throughput, they may individually experience long sequences of transmissions in rapid succession, interspersed with extended periods of starvation. These starvation effects are particularly pronounced when the nodes become more aggressive in trying to activate, which is the regime in which the network should operate to achieve maximum throughput.

Most of the literature focuses on single-channel CSMA algorithms, in which all the transmission occur on the same frequency. In this paper we consider the natural generalization in which multiple orthogonal frequency channels are available; various variations of this multi-channel CSMA model have been studied in~\cite{AZ13,Andrews2017,Block2016,BF12,LCCL12,LZCC10,PYLC10}. Our ultimate goal is understanding whether temporal starvation effects can be effectively mitigated by the usage of multiple channels, as suggested in~\cite{LCCL12}.

In order to make a fair comparison, we do not assume that additional channels can be added to the existing one, but rather that the total available wireless spectrum could be divided into $\nc$ non-overlapping channels with smaller capacity on which nearby nodes can transmit without interfering with each other. In this way, at the cost of a potentially lower throughput, the starvation effects in the network can be alleviated or even eliminated and in this way obtain substantial improvements in delay performance. The aim of the present paper is to investigate this non-trivial trade-off and for this reason we analyze the aggregate throughput of multi-channel CSMA networks and develop an analytical framework for quantifying temporal starvation effects for these networks.

The temporal starvation in the context of multi-channel CSMA networks has been mostly studied by means of mixing times~\cite{LCCL12}. The large majority of these results, however, assumes either that the activation rate is very small or that there are enough channels for all the nodes to be active simultaneously, so that existing results for multi-color Glauber dynamics can be exploited. 
The first crucial difference of our approach is that we are mostly interested in the scenario in which the number of available channels does not allow simultaneous activity of all the nodes. This creates a complex network dynamics in which nodes still compete for transmission and phenomena like temporal starvation and spatial unfairness persist.
Furthermore, instead of focusing on mixing times, inspired by~\cite{Kai2015}, we study temporal starvation by means of hitting times of the Markov process describing the CSMA dynamics.

We focus on the regime in which the activation rate grows large, in which the CSMA dynamics favors the activity states with a maximum number of active nodes, to which we will refer as \textit{dominant states}. These activity states play a crucial role in our analysis, since the timescales at which transitions between them occur are intimately related to the magnitude of the temporal starvation effects.

Our contributions can be summarized as follows.
\begin{itemize}
	\item We show that the multi-channel CSMA dynamics can be represented as single-channel CSMA dynamics on an appropriate virtual network, even when the possible interference on the different channels is described by different conflict graphs. This equivalent representation allows us to immediately derive the stationary distribution of multi-channel CSMA dynamics and to prove its insensitivity to back-off and transmission time distributions.
	\item We prove various properties of the asymptotic aggregate throughput and, in particular, that is a non-increasing function of the number $\nc$ of available channels.
	\item We show how the temporal starvation can be evaluated in the high-activation limit by studying the expected transition times between the dominant activity states. Furthermore, we characterize the timescale of temporal starvation phenomena in terms of the structure of the state space of the Markov process describing the CSMA dynamics.
	\item We analyze the mixing time of multi-channel CSMA dynamics using the same framework built to study the transition times between dominant states. Our analysis suggests that in the high-activation regime the mixing time does not always correctly characterize the temporal starvation timescale in multi-channel CSMA networks, often leading to pessimistic performance estimates.
	\item By means of various counterexamples, we show that many desirable properties and performance indices are not monotone in the number $\nc$ of available channels, revealing the difficulty of finding analytically the best trade-off between throughput and temporal starvation effects.
\end{itemize}

The rest of the paper is organized as follows. The CSMA network models for both the single-channel and multiple-channel case are described in Section~\ref{sec2}. The equivalent representation for multi-channel CSMA networks is presented in Section~\ref{sec3}. In Section~\ref{sec4} we introduce the notion of dominant activity states and derive the properties of the aggregate throughput. We show in Section~\ref{sec5} how we characterize temporal starvation by looking at asymptotic results for transition times between dominant states and relate it to structural property of the underlying state space. Section~\ref{sec6} is entirely devoted to mixing times and to their relation with the worst temporal starvation timescale.
Lastly, in Section~\ref{sec7} we discuss some generalizations of our model and show how our framework extends to them.

\section{Model description}
\label{sec2}
In this section we introduce the stochastic models that describe the dynamics of random-access networks operating according to CSMA-like algorithms. We present first the model for the case of single-channel CSMA algorithm and later its generalization to the multi-channel scenario. The two models are presented separately for notational convenience especially since in the next section we will show how the multi-channel model has an equivalent representation as single-channel model.

\subsection{Single-channel CSMA network}
\label{sub21}
We consider a network of transmitter-receiver pairs sharing a wireless medium according to a CSMA-type algorithm. A \textit{node} indicates potential data transmission between a transmitter and a receiver. 


Every node can either be active or inactive, depending on whether the data transmission is ongoing or not. We assume that the network consists of $N$ such nodes, so that the network activity state can then be described by an $N$-dimensional vector $x$, where $x_i=1$ if node $i$ is active and $x_i=0$ otherwise.

We assume that the network structure and interference conditions can be described by means of an undirected finite graph $G=(V,E)$, called \textit{conflict graph}, where the set of vertices $V=\{1,\dots,N\}$ represents the nodes of the network and the set of edges $E \subseteq V \times V$ indicate which pairs of nodes cannot be active simultaneously. Therefore, neighboring nodes in the conflict graph are prevented from simultaneous activity by the carrier-sensing mechanism. 


We focus on the scenario where nodes are saturated, which means that nodes always have packets available for transmission and it is particularly relevant in high-load regimes. The transmission times of node $i$ are independent and exponentially distributed with mean $1/\mu_i$. When the transmission of a packet is completed, node $i$ \textit{deactivates} (i.e., releases the medium) and starts a back-off period. The back-off periods of node $i$ are independent and exponentially distributed with mean $1/\nu_i$. At the end of each back-off period node $i$ \textit{activates} (i.e., starts the transmission of a new packet) if and only if all its neighboring nodes on $G$ are currently inactive.

Let $\cX \subseteq \{0,1\}^N$ be the set of all feasible joint activity states of the network. Since the interference is modeled by the conflict graph $G$, the set $\cX$ consists of the incidence vectors of all independent sets of the conflict graph $G$:
\[
	\cX := \Bigl \{ x \in \{0,1\}^N ~:~ x_i x_j=0 \, \, \forall \, (i,j) \in E \Bigr \}.
\]
If we let $X(t) \in \cX$ denote the network activity state at time $t$, then the CSMA dynamics is described by a continuous-time Markov process $\{X(t)\}_{t \geq 0}$ on the state space $\cX$ with transition rates between $x,y \in \cX$ given by
\[
	q(x,y):=
	\begin{cases}
		\nu_i & \text{if } y = x + e_i \in \cX,\\
		\mu_i & \text{if } y = x - e_i \in \cX,\\
		0 & \text{otherwise,}
	\end{cases}
\]
where $e_i \in \{0,1\}^N$ is the vector with all zeros except for a $1$ in position $i$. The Markov process $\{X(t)\}_{t\geq 0}$ is reversible~\cite{BKMS87} and has a product-form stationary distribution
\begin{equation}
\label{eq:statdist}
	\pi(x) := Z^{-1} \prod_{i=1}^N \Bigl (\frac{\nu_i}{\mu_i} \Bigr)^{x_i}, \quad x \in \cX,
\end{equation}
where $Z$ is the normalizing constant 
\[
	Z := \sum_{x \in \cX}  \prod_{i=1}^N \Bigl (\frac{\nu_i}{\mu_i} \Bigr)^{x_i}.
\]
The stationary distribution~\eqref{eq:statdist} is insensitive to the distributions of back-off periods and transmission times, in the sense that it depends on these only through their averages $1/\nu_i$ and $1/\mu_i$, as proved in~\cite{vdVBvLP10}. Hence, \eqref{eq:statdist} holds in fact for general back-off and transmission time distributions.

\subsection{Multi-channel CSMA network}
\label{sub22}
In this paper we consider a generalization of the saturated single-channel CSMA model described in Subsection~\ref{sub21} where each node can sense the interference and transmit on any of the $\nc$ available channels, but at most on one at a time.
We assume that for every $c=1,\dots,\nc$ all possible conflicts between nodes on channel $c$ are described by a conflict graph $G_c=(V,E_c)$. Node $i$ in the network has a different back-off timer for each of the $\nc$ available channels and we model these timers as $\nc$ independent Poissonian clocks, ticking at rates $\nu_{i,1},\dots, \nu_{i,\nc}$. When the first of these $\nc$ clocks rings for an inactive node, it activates on the corresponding channel, say $c$, if and only if the neighboring nodes of $i$ in $G_c$ are not active on the same channel. The transmission times of node $i$ on channel $c$ are independent and exponentially distributed with mean $1/\mu_{i,c}$.

A network activity state is described by a vector $x \in \{0,1,\dots, \nc\}^N$ where $x_i=0$ if node $i$ is inactive and $x_i = c$ if node $i$ is active on channel $c$ with $ 1 \leq c \leq \nc$. Let $\cX_C$ be the collection of feasible network activity states, which consists of all the vectors $x \in \{0,1,\dots, \nc\}^N$ such that for every $c=1,\dots,C$ two neighboring nodes in $G_c$ are not simultaneously active on channel $c$, \ie
\begin{equation}
\label{eq:admissiblestatesmultichannel}
	{\small \cX_C :=\{ x \in \{0,1,\dots,\nc\}^N  \,: \, \forall \, c \, \forall \, (i,j) \in E_c \, x_i x_j =0 \text{ or } x_i \neq x_j \}.}
\end{equation}
If the vector $X_C(t) \in \cX_C$ describes the activity state of the network at time $t$, then $\{X_C(t)\}_{t \geq 0}$ is a Markov process on the state space $\cX_C$ with transition rates between $x,y \in \cX_C$ given by
\begin{equation}
\label{eq:multichannel_rates}
	q_C(x,y)=
	\begin{cases}
		\nu_{i,c} & \text{ if } y = x + c \cdot e_i \in \cX_C,\\
		\mu_{i,c} & \text{ if } y = x - c \cdot e_i \in \cX_C,\\
		0 & \text{ otherwise.}
	\end{cases}
\end{equation}
Later we will briefly discuss a generalization of this multi-channel CSMA model in which each node can possibly be simultaneously active on more than one channel and discuss which of our results extend to this setting, see Section~\ref{sec7}.

\section{An equivalent representation for multi-channel CSMA networks}
\label{sec3}
In this section we argue that the evolution of a network using a multi-channel CSMA algorithm can be equivalently described as a \textit{virtual network} operating under the single-channel CSMA algorithm on a modified conflict graph $G^*$. 

Consider the undirected graph $G^*=(V^*,E^*)$ with vertex set $V^*:= V \times \{1,\dots,\nc\}$ where two nodes $(i,c)$ and $(i',c')$ are adjacent if and only if
\[
	\begin{cases}
		c \neq c',\\
		i = i',	
	\end{cases}
	\quad \text{ or } \quad
	\begin{cases}
		c=c',\\
		(i,i') \in E_c.	
	\end{cases}
\]
To avoid confusion with the original nodes, we refer to the nodes of $G^*$ as \textit{virtual nodes}. The activity state of the virtual network is then represented by a $0$-$1$ vector of length $\nc \cdot N$ and we denote by $\cX^* \subset \{0,1\}^{ \nc \cdot N}$ the set of admissible virtual activity states on $G^*$. 

Define the function $\mathcal V: \cX^* \to \cX_C$ that maps a virtual network state $\mathbf{x} \in \cX^*$ in the network state $\mathcal V(\mathbf{x}) \in \cX_C$ such that for every $i=1,\dots, N$
\[
	\mathcal V(\mathbf{x})_i =
	\begin{cases}
		c & \text{if } \exists \, c \text{ s.t. } \mathbf{x}_{(i,c)}=1,\\
		0 & \text{otherwise}.
	\end{cases}
\]
The function $\mathcal V$ is well defined, since for every $c\neq c'$ the virtual nodes $(i,c)$ and $(i,c')$ are neighbors in $G^*$ and thus only one of them can be active. Furthermore, it is easy to check that $\mathcal V$ is a bijection, from which the following lemma immediately follows.

\begin{lemma}
The collection $\cX^*$ of activity states on the virtual network $G^*$ is in one-to-one correspondence with the collection $\cX_C$ of multi-channel activity states on $G$ introduced in~\eqref{eq:admissiblestatesmultichannel}.
\end{lemma}

Suppose further that the $\nc \cdot N$ virtual nodes in $G^*$ operate using the single-channel CSMA algorithm described in Subsection~\ref{sub21}, assuming that the virtual node $(i,c)$ has back-off periods that are exponentially distributed with mean $1/ \nu_{i,c}$ and transmission periods that are exponentially distributed with mean $1/ \mu_{i,c}$. The virtual network activity evolves then as a continuous-time Markov process on $\cX^*$, which we denote as $\{\widetilde{X}(t)\}_{t\geq 0}$.

It is easy to check that the transition rate between any two virtual network states $\mathbf{x}, \mathbf{y} \in \cX^*$ is the same as the transition rate given in~\eqref{eq:multichannel_rates} between the multi-channel activity states $\mathcal V(\mathbf{x}), \mathcal V(\mathbf{x}) \in \cX_C$. Hence, the processes $\{X_C(t)\}_{t \geq 0}$ and $\{\widetilde{X}(t)\}_{t\geq 0}$ are two representations of the same Markovian dynamics, and if $X_C(0) = \mathcal V(\widetilde{X}(0))$, then
\[
	X_C(t) \ed \mathcal V(\widetilde{X}(t)), \quad \forall \, t \geq 0.
\]

This equivalent representation can be used in combination with~\eqref{eq:statdist} to obtain the stationary distribution $\pi_C$ of the multi-channel activity process $\{X_C(t)\}_{t \geq 0}$, as illustrated by the following proposition.

\begin{proposition}
The multi-channel network activity process $\{X_C(t)\}_{t \geq 0}$ has product-form stationary distribution
\begin{equation}
\label{eq:statdist_multichannel}
	\pi_C(x) = Z^{-1} \prod_{i=1}^{N} \prod_{c=1}^{\nc}\left ( \frac{\nu_{i,c}}{\mu_{i,c}} \right )^{\mathds{1}_{\{x_i=c\}}}, \quad x \in \cX_C,
\end{equation}
where $Z$ is the appropriate normalizing constant. 
\end{proposition}
We remark that multi-channel CSMA dynamics was already shown in~\cite{BF12} to have a product-form distribution and, in the special case where the interference is described by the same conflict graph on every channel $E_1=\dots=E_{\nc}$, also in~\cite{PYLC10}. Our proof approach based on this equivalent representation recovers the same result in the most general case and immediately shows that the insensitivity property carries over to the stationary distribution~\eqref{eq:statdist_multichannel} of the multi-channel CSMA dynamics, which thus holds for general back-off and transmission time distributions.

\section{Dominant states and aggregate throughput}
\label{sec4}
In this section we study how the number of channels $\nc$ affects the performance in terms of aggregate throughput. While increasing the number of channels evidently provides greater transmission opportunities, the net impact on throughput performance is non-obvious since the transmission capacity per channel is inversely proportional to the number of channels. 

Consider a multi-channel CSMA network as described in Subsection~\ref{sub22} and assume further that each of these channels has capacity $1/\nc$ and that all the nodes have homogeneous activation and transmission rates, namely
\[
	\nu_{i,c} \equiv \nu \, \text{ and } \, \mu_{i,c} \equiv \mu \quad \forall \, \,c=1,\dots,\nc, \quad \forall \, i =1,\dots,N.
\]
Without loss of generality, we henceforth also assume that $\mu=1$. The stationary distribution~\eqref{eq:statdist_multichannel} of the activity process $\{X_C(t)\}_{t \geq 0}$ then reads
\begin{equation}
\label{eq:statdist_multichannel_homogeneous}
	\pi_C(x) = Z^{-1} \prod_{i=1}^{N} \nu^{\ind{x_i \neq 0}}= Z^{-1} \nu^{a(x)}, \quad x \in \cX_C, 
\end{equation}
where $a(x):=\sum_{i=1}^N \ind{x_i \neq 0}$ is the number $a(x)$ of active nodes in $x$, regardless of the channels they are active on. As~\eqref{eq:statdist_multichannel_homogeneous} shows, when $\nu >1$, the stationary probability of an activity state $x \in \cX_C$ increases with its cardinality in an exponential fashion.

Define $\ac$ to be the maximum number of active nodes in conflict graph $G$ when $C$ channels are available, \ie
\[
	\ac:=\max_{x \in \cX_C} \sum_{i \in V} \mathds{1}_{\{ x_i \neq 0 \}} = \max_{x \in \cX_C} a(x).
\]
In this regime, the stationary distribution~\eqref{eq:statdist_multichannel_homogeneous} favors the activity states with a maximum number of active nodes. We refer to such activity states as \textit{dominant (activity) states} and denote by $\cD$ their collection, \ie
\[
	\cD := \arg\max_{x \in \cX_C} a(x) = \{ x \in \cX_C : a(x)=\ac\}.
\]

In the rest of the section, we assume further that the same conflict graph $G=(V,E)$ describes the interference on all $\nc$ channels, allowing us to draw a parallel between network activity states and colorings of the graph $G$. 

In order to characterize the dominant states of a multi-channel CSMA network, we need some further definitions. A \textit{vertex coloring} of the graph $G$ is a labelling of the graph’s vertices with colors such that no two vertices sharing the same edge have the same color. A \textit{(proper) $C$-coloring} is a vertex coloring using at most $C$ colors, see \eg \cite{B89,C89}. The smallest number of colors needed to color a graph $G$ is called its \textit{chromatic number}, and is denoted as $\chi(G)$. 

If $C=1$, we are back in the single-channel scenario and the activity states are in one-to-one correspondence with the independent sets of the graph $G$. In particular, $\mathcal{A}(1)$ is equal to the cardinality $\alpha(G)$ of the maximum independent set on $G$, often referred to as the \textit{independence number} of the graph $G$. 

For $C \geq 2$, we showed in Section~\ref{sec3} that the multi-channel activity on $G$ is equivalent to the single-channel activity on the virtual network $G^*$. Since we assumed $E_1=\dots=E_C$, the virtual network $G^*$ is the Cartesian product of $G$ and the complete graph with $C$ nodes. In this scenario $\ac=\alpha(G^*)$, where $\alpha(G^*)$ is the independence number of the graph $G^*$, as proved~\cite[Lemma 1]{B89}.

If the number of available channels in a CSMA network is larger than or equal to the chromatic number of the corresponding conflict graph, \ie $\nc \geq \chi(G)$, then every proper $\nc$-coloring of the graph $G$ corresponds to a dominant state for the network dynamics and, in particular, $\ac=N$. 

If instead $\nc < \chi(G)$, by definition of the chromatic number, a proper $\nc$-coloring of the graph $G$ does not exist and, in particular, there are no admissible activity states where all nodes are active. All admissible activity states correspond to \textit{partial $\nc$-colorings} of the graph $G$. However, the problem of finding the maximum number of nodes that can be active simultaneously in a general conflict graph $G$ with $\nc$ available channels is non-trivial, since it is at least as hard as finding the maximum independent sets of $G^*$. 

Since nodes are saturated, the \textit{throughput} $\theta_i(C)$ of node $i$ is proportional to the long-run fraction of time that node $i$ is active and can be expressed as
\begin{equation}
\label{eq:asythroughput}
	\theta_i (C):= \frac{1}{C} \sum_{x \in \cX_C} \pi_C(x) \mathds{1}_{\{ x_i \neq 0 \}},
\end{equation}
where the constant $1/C$ accounts for the fact that each channel has capacity $1/C$. Being a function of the stationary distribution, the throughput is also insensitive to the distributions of the back-off and transmission times.

We analyze the \textit{aggregate throughput} $\Theta(C)$ of a multi-channel CSMA network with $C$ available channels in the asymptotic regime where the activation rate $\nu$ grows large, defined as
\[
	\Theta(C) := \lim_{\ninf} \sum_{i \in V} \theta_i(C) = \lim_{\ninf} \sum_{i \in V} \frac{1}{C} \sum_{x \in \cX_C} \pi_C(x) \mathds{1}_{\{ x_i \neq 0 \}}.
\]
Even if it may be infeasible to calculate $\ac$ for a given random-access network using $C$ channels, some conclusions for its aggregate throughput $\Theta(C)$ can still be drawn, as illustrated by the following theorem. 

\begin{theorem}[Aggregate throughput]\label{thm:throughput}
The following statements hold for the aggregate throughput $\Theta(C)$ of a multi-channel CSMA network where the same conflict graph $G=(V,E)$ describes the interference on all $\nc$ channels:
\begin{itemize}
	\item	[(i)] $\displaystyle \Theta(C) = \ac/C$ for every $C\geq 1$;
	\item	[(ii)] $\Theta(C)=\alpha(G)$ for every $1 \leq C \leq C^*(G)$, where $C^*(G)$ is the number of disjoint maximum independent sets of the graph $G$ ;
	\item	[(iii)] $\displaystyle \Theta(C)=N/C$ for every $C \geq \chi(G)$;
	\item	[(iv)] $\Theta(C)$ is a non-increasing function of the number of channels $C$.
\end{itemize}
\end{theorem}
\begin{proof}
The definition of dominant state and~\eqref{eq:statdist_multichannel_homogeneous} yield that $\sum_{x \in \cD} \pi_C(x) \to 1$ as $\ninf$, and thus
\begin{align*}
\Theta(C)
	 = \lim_{\ninf} \sum_{i \in V} \frac{1}{C} \sum_{x \in \cX_C} \pi_C(x) \, \ind{x_i \neq 0} 
	 = \frac{1}{C} \lim_{\ninf} \sum_{x \in \cX_C} \pi_C(x) \sum_{i \in V} \ind{x_i \neq 0}
	 = \frac{\ac}{C} \lim_{\ninf}  \sum_{x \in \cD} \pi_C(x)  = \frac{\ac}{C}.
\end{align*}
As mentioned earlier, when $C \geq \chi(G)$ there exists a proper $C$-coloring of the graph $G$, which means all nodes can be simultaneously active and thus $\ac=N$. Furthermore, it is easy to prove by induction that $\ac=C \cdot \alpha(G)$ for every $C \leq C^*(G)$. To prove (iv), we will show that the following inequality holds
\[
	\frac{\mathcal{A}(C+1)}{C+1} \leq \frac{\ac}{C}, \quad \forall \, C \geq 1.
\]
The inequality trivially holds when $C \geq \chi(G)$, since $\ac=N$. The proof in the case $C \leq \chi(G)$ readily follows by considering a $(C+1)$-coloring that yields $\mathcal{A}(C+1)$, pick the color that is used the least and discolor the corresponding nodes, which are at most $\lfloor \frac{\mathcal{A}(C+1)}{C+1} \rfloor$, obtaining a $C$-coloring with $\mathcal{A}(C+1) - \lfloor \frac{\mathcal{A}(C+1)}{C+1} \rfloor \geq \frac{C}{C+1} \mathcal{A}(C+1)$ colored nodes. Using this newly obtained $C$-coloring we immediately get that $\mathcal{A}(C) \geq \frac{C}{C+1} \mathcal{A}(C+1)$.
\end{proof}
%
%
%

In view of the fact that $\Theta(C)$ is a non-increasing function of the number of channels $C$, it seems that the network should be operated using a single channel, but this is also the scenario where temporal starvation phenomena and spatial unfairness are often more pronounced. 

Jain's fairness index~\cite{Jain1984} is often used as a fairness measure in the context of random-access algorithms and we consider here its asymptotic counterpart $J(C)$ in the high-activation limit $\ninf$, \ie
\[
	J(C):= \limnu \frac{\left ( \sum_{i=1}^N \theta_i(C)\right )^2}{N \sum_{i=1}^N \theta_i(C)^2}.
\]
One may conjecture that a larger number of available channels should always reduce the spatial unfairness in the network. This is indeed true for many conflict graphs, but does not hold in general. Figure~\ref{fig:counterexample_jain} shows a particular example of a network in which by adding an additional channel, the spatial unfairness worsens as $J(2)<J(1)$.

\begin{figure}[!h]
\centering
	\subfloat[The three dominants states in $\mathcal{D}_1$, which yield $J(1) = 9/13 \approx 0.692$]{\includegraphics[scale=1.1]{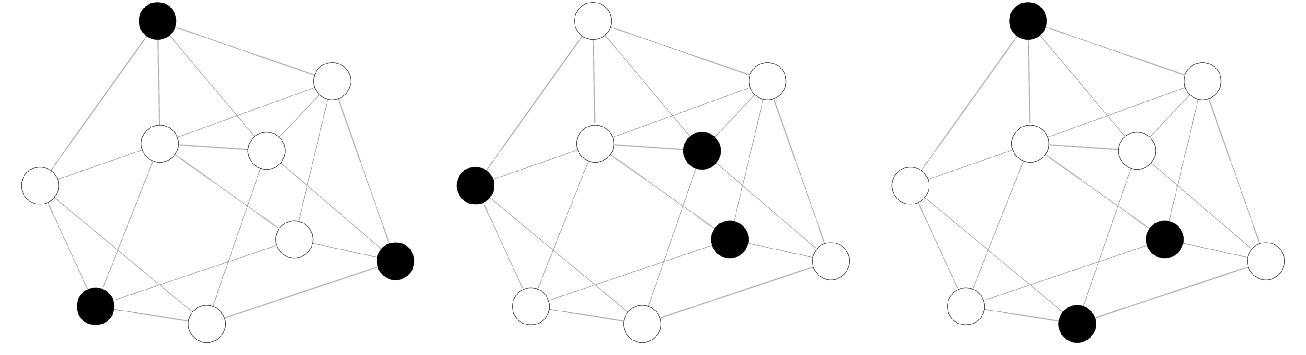}}\\
	\subfloat[The two dominants states in $\mathcal{D}_2$, which yield $J(2) = 2/3 \approx 0.667$]{\includegraphics[scale=1.1]{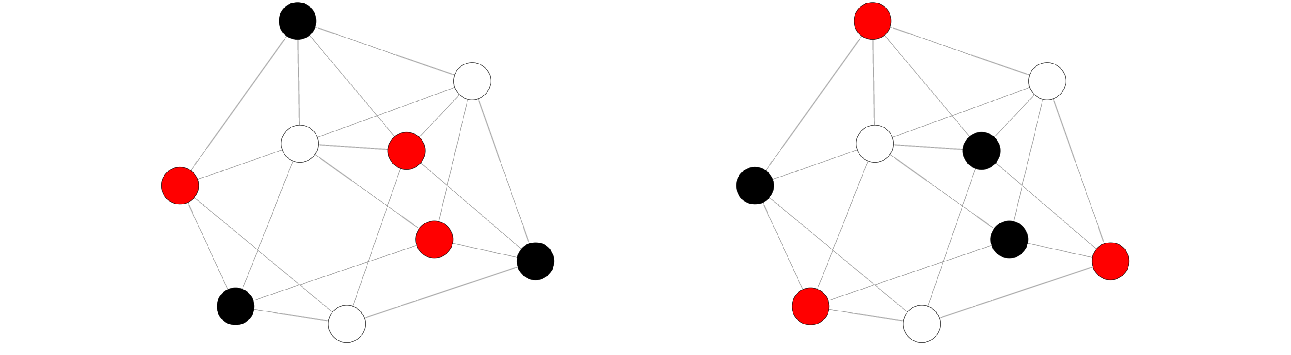}}
	\caption{Example of a network and its dominant states with $C=1$ and $C=2$ channels available, respectively.}
	\label{fig:counterexample_jain}
\end{figure}


As far as temporal starvation is concerned, we expect that by exploiting more channels the temporal starvation effects can be mitigated. However, making this dependence more explicit is a challenging task, since the structure of the conflict graph plays a crucial role via the parametric family $\cD$ of dominant activity states. In general, a non-trivial trade-off emerges between the magnitude of temporal starvation effects and the aggregate throughput that the network can achieve and a balance can be found using a number of channels in the range between $1$ and $\chi(G)$.

\section{Temporal starvation evaluation via transition times}
\label{sec5}
In this section, we introduce a framework to study and quantify temporal starvation effects in multi-channel CSMA networks in the high-activation regime $\ninf$.

When the activation rate grows large, the network spends roughly the same fraction of time in each dominant state. However, it takes a long time for the activity process to move from one dominant state to the other, since such a transition involves the occurrence of rare events. Indeed, intuitively, the activity process must follow a transition path through some activity states with fewer active nodes which are highly unlikely in view of~\eqref{eq:statdist_multichannel_homogeneous} and the time to reach such activity states is correspondingly long.

We study the transitions between dominant states in terms of first hitting times of the activity process $\{X_C(t)\}_{t\geq 0}$. In the limit $\ninf$ the asymptotic behavior of the transition time between any pair of dominant states can be characterized exactly on a logarithmic scale by studying the structural properties of state space $\cX_C$.

The main idea of our method is inspired by the concept of ``traps'' of the state space introduced in the context of single-channel CSMA networks in~\cite{KL11,Kai2015}. Traps are essentially subsets of activity states where the activity process remains for a long time before leaving. By focusing on the asymptotic regime $\ninf$, we can restrict our analysis only to the traps corresponding to dominant states and, moreover, characterize the timescale of the transitions between them without the need of a detailed description of the structure of these traps.

In this section, we consider the multi-channel CSMA dynamics with homogeneous rates as in Section~\ref{sec4} but 
here we do \textit{not} need to assume that the interference structure on different channels is described by the same conflict graph $G$. Note that the assumption of homogeneous rates across different nodes and channels is not crucial and in fact it can be relaxed, but to keep the exposition simple we do not present here the general case and the interested reader is referred to Section~\ref{sec7}, where we illustrate also other generalizations of our model for which this framework is still valid.

Let $\tau_{x,A}(\nu):=\inf\{t >0 : X_C(t) \in A \}$ be the \textit{first hitting time} of the subset $A \subset \cX_C$ for the Markov process $\{X_C(t)\}_{t\geq 0}$ starting in $x$ at $t = 0$.
For any $x,y \in \cX_C$, we say that $\o \subset \cX_C $ is a \textit{path} from $x$ to $y$ in $\cX_C$, and denote it by $\o: x \to y$, if $\o$ is a finite sequence of states $\o_1,\dots,\o_n \in \cX_C$ such that $\o_1=x$, $\o_n = y$ and the CSMA dynamics allows the step from $\o_i$ to $\o_{i+1}$ for every $i=1,\dots,n-1$. The \textit{communication height} between two states $x,y \in \cX_C$ is the minimum number of nodes (with respect to that of any dominant state) that need to become simultaneously inactive at some point along any path from $x$ to $y$, \ie
\begin{equation}
\label{eq:ch}
	\D(x,y):=\min_{\o: x \to y} \max_{z \in \o} (\ac-a(z)). 
\end{equation}
Both the notions of a path and communication height naturally extends to the case of two subsets $A, B \subset \cX_C$.

The communication height is a crucial quantity in our analysis since in the regime $\ninf$ the most likely trajectories from $x$ to $y$ are precisely those that achieve the min-max in~\eqref{eq:ch}. Indeed, all the other trajectories $\o: x \to y$ will visit at some point an activity state with fewer active nodes and thus correspondingly more rare in view of~\eqref{eq:statdist_multichannel_homogeneous}. The communication height $\D(\cdot,\cdot)$ defines an \textit{ultrametric} on $\cX_C$, since
\begin{enumerate}
	\item $\D(x,y) \geq 0$ and equality holds if and only if $x=y$;
	\item $\D(x,y)=\D(y,x)$;
	\item $\D(x,y) \leq \max\{ \D(x,z), \D(z,y)\}$.
\end{enumerate}
The second statement in property (1) holds since if $x\neq y$ at every step of the CSMA dynamics the number of active nodes must either increase or decrease by $1$ and cannot stay constant along any path from $x$ to $y$ and, hence, $\D(x,y)>0$.

For any $a\in \R^+$, we denote by $\log_\nu a$ the logarithm in base $\nu$. Since most of the results are in the limit $\ninf$, we will henceforth assume that $\nu > 1$, so that $\log_\nu(\cdot)$ is a strictly increasing function from $\R^+$ to $\R$.
The next theorem shows that the communication height $\D(s,D)$ between a dominant state $s \in \cD$ and a subset $D \subseteq \cD \setminus \{s\}$ completely characterizes on a logarithmic scale the expected transition time from $s$ to $D$ in the limit $\ninf$.

\begin{theorem}[Timescale of transitions between dominant states] \label{thm:hit}
For every dominant state $s \in \cD$ and target subset $D \subseteq \cD \setminus \{s\}$ of dominant states different from $s$, the following asymptotic equality holds
\[
	\limnu \log_\nu \E \tau_{s,D}(\nu) = \D(s,D)-1.
\]
\end{theorem}
The proof of this result is presented in Subsection~\ref{sub51}. We now illustrate how Theorem~\ref{thm:hit} allows us to identify the longest timescale for which each node of the network starves. For each node $i \in V$ let $\cS_C(i):=\{ s \in \cD : s_i \neq 0 \}$ be the subset of dominant states in which node $i$ is active. It is natural to study the temporal starvation only for the nodes which are active in at least one dominant states, \ie $\cS_C(i) \neq \emptyset$, and that are not active in every dominant states, \ie $\cS_C(i) \neq \cD$. For every such node $i$ we define its \textit{starvation index} as
\begin{equation}
\label{eq:siv}
	\siv:= \max_{s \in \cD \setminus \cS_C(i)} \min_{s' \in \cS_C(i)} \D(s,s').
\end{equation}
From property (1) of the communication heights it immediately follows that $\siv \geq 1$. We choose not to define the starvation index of the nodes $i\in V$ such that $\cS_C(i) = \emptyset$, since these nodes do not properly suffer from \textit{temporal} starvation, since they permanently starve in the limit $\ninf$ having asymptotically zero throughput. Similarly, $\siv$ is not defined for any node $i$ that is active in every dominant state of the network, \ie $\cS_C(i) = \cD$, since any such node $i$ does not starve in view of the fact that $\sum_{x \in \cX_C} \pi_C(x) \ind{x_i \neq 0} \to 1$ as $\ninf$.

The starvation index $\siv$ captures the largest timescale at which node $i$ suffers from temporal starvation at stationarity in the high-activation regime $\ninf$. Indeed, $\siv$ characterizes on a logarithmic scale how long it takes the activity process to reach the subset $\cS_C(i)$ of dominant states where node $i$ is active when starting in the worst possible dominant state in $\cD \setminus \cS_C(i)$, as illustrated by the following identity
\[
	\limnu \log_\nu \Big ( \max_{s \in \cD \setminus \cS_C(i)} \E \tau_{s,\cS_C(i)}(\nu) \Big ) = \siv -1.
\]
Lastly, we define the \textit{starvation index of the network} as the worst starvation index of its nodes, \ie
\begin{equation}
\label{eq:si}
	\si := \max_{i \in V : \cS_C(i) \neq \emptyset, \cD} \siv.
\end{equation}

The index $\si$ is a non-increasing function of the number of available channels for all networks of which we analyzed numerically the corresponding state spaces $\cX_1,\dots,\cX_{\chi(G)-1}$. However, proving this claim analytically is challenging in view of the intricate nature of the dominant states. Indeed, dominant states in $\mathcal{D}_{\nc+1}$ are not necessarily obtained from dominant states in $\mathcal{D}_{\nc}$ by activating a maximum number of inactive nodes on the $C+1$-th channel, as illustrated for the network in Figure~\ref{fig:difficult}. Hence, there is no simple relation between the sets $\cD$ and $\mathcal{D}_{\nc+1}$ nor any other ``monotonicity property'' that one could leverage. In particular, this means that the maximum in~\eqref{eq:si} may have to possibly be taken over different and/or non-nested subsets of nodes when increasing $\nc$.
\begin{figure}[!h]
\centering
	\includegraphics[scale=1.1]{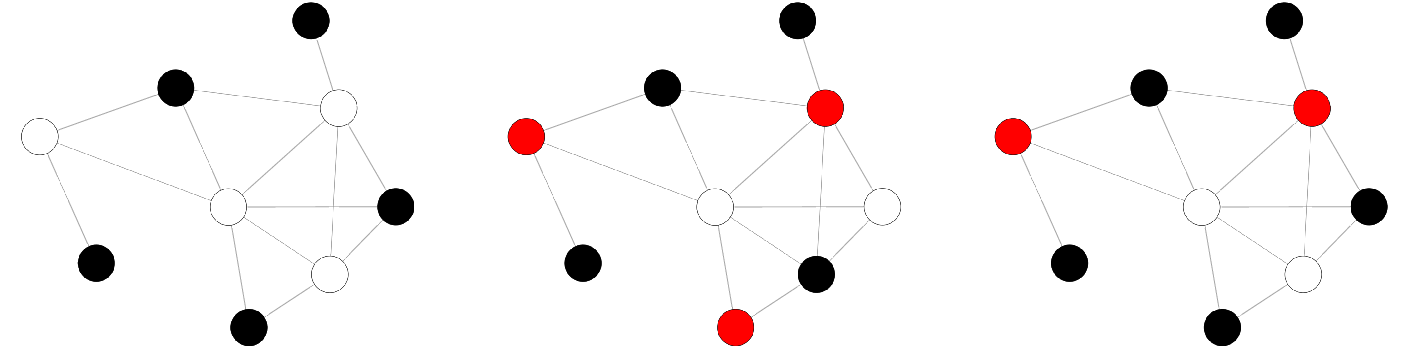}
	\caption{Example of a network with the unique dominant state $s^{(1)} \in \mathcal{D}_1$ (on the left) and two dominant states $s^{(2)}, s^{(3)} \in \mathcal{D}_2$ (in the middle and on the right). In particular, the dominant state $s^{(2)}$ is not superset of the dominant state $s^{(1)}$.}
	\label{fig:difficult}
\end{figure}


\subsection{Asymptotics for transition times between dominant states}
\label{sub51}
This subsection is devoted to the proof of Theorem~\ref{thm:hit}. We first introduce the uniformized discrete-time version of the activity process, which will allow us to use classical results for expected hitting times of reversible Markov chains and their analogy with electrical networks, see \eg \cite[Section 2.1]{Hollander2017a} and references therein. 

In this subsection we will almost entirely suppress the subscript $C$ for notational compactness, but the reader should keep in mind that the results hold for a general multi-channel CSMA dynamics as described in Subsection~\ref{sub22}.
We consider the uniformized discrete-time version of the activity process $\{X_C(t)\}_{t\geq 0}$. In more detail, we construct a discrete-time Markov chain $\{\widetilde{X}(t)\}_{t \in \N}$ starting from the continuous-time Markov process $\{X_C(t)\}_{t\geq 0}$ by means of uniformization at rate $q_{\mathrm{max}}:=\max_{x \in \cX_C} \sum_{y \neq x} q(x,y) = C N \nu$.
The transition probabilities of this Markov chain read
\[
	P(x,y) =
	\begin{cases}
		\frac{1}{C N} \nu^{-[a(x)-a(y)]^+} & \text{ if } d(x,y)=1,\\
		0 & \text{ if } d(x,y)>1,\\
		1 - \sum_{z \in \cX, z\neq x} P(x,z) & \text{ if } x=y,
	\end{cases}
\]
where $d(x,y)$ is the distance function on $\cX \times \cX$ defined as
\[
	d(x,y):=|\{ i \in V : x_i \neq y_i\}|+|\{ i \in V : x_i \neq y_i \text{ and } x_i y_i \neq 0 \}|.
\]
The quantity $d(x,y)$ represents the minimum number of steps that the CSMA dynamics needs to go from state $x$ to state $y$; the second term on the right-hand side accounts for the fact that if a node is active on a channel $c$, the immediate activation on another frequency $c'\neq c$ is not allowed and requires first the node to become inactive. 

This distance function on $\cX$ reflects the main difference between the multi-channel CSMA dynamics and the classical multi-color Glauber dynamics, in which a node can change channel/color in a single step (see~\cite[Section IV.C]{Lam2012} for more details). The two corresponding state spaces, while still having the same states, have thus fundamentally different structures and, in particular, there are many more possible trajectories between dominant states in the one corresponding to the Glauber dynamics. This fact means that using Glauber dynamics to study temporal starvation phenomena may lead to overly optimistic performance bounds.

It is easy to check that~\eqref{eq:statdist_multichannel_homogeneous} is the stationary distribution also for the uniformized chain $\{\widetilde{X}(t)\}_{t \in \N}$. Let $T_{x,A}(\nu):=\inf\{t \in \N : \widetilde{X}(t) \in A \}$ be the discrete-time counterpart of the first hitting time $\tau_{x,A}(\nu)$. These hitting times are closely related as
\[
		\tau_{x,A}(\nu) \stackrel{d}{=} \sum_{i=1}^{T_{x,A}(\nu)} \mathcal E_i(\nu),
\]
where $\{ \mathcal E_i(\nu) \}_{i\in \N}$ are i.i.d.~exponential random variables with mean $q_{\mathrm{max}}^{-1}=( C N \nu)^{-1}$. In particular, it holds that
\begin{equation}
\label{eq:contdisc}
	\E \tau_{x,A}(\nu) = (C N \nu)^{-1} \, \E T_{x,A}(\nu).
\end{equation}
This equality shows that we can study the expected transition times between dominant states in the discrete-time setting where the theory is richer and then immediately translate them back for the continuous-time activity process $\{X_C(t)\}_{t\geq 0}$. We are now ready to present the proof of the main result of this section.

\begin{proof}[Proof of Theorem~\ref{thm:hit}]
We first introduce the following notation to asymptotically compare two functions $f,g: \R \to \R$: we write $f(\nu) \prec g(\nu)$ if $f(\nu) = o(g(\nu))$, $f(\nu) \preceq g(\nu)$ if $f(\nu) = O(g(\nu))$ as $\ninf$ and $f(\nu) \asymp g(\nu)$ if both $f(\nu) \preceq g(\nu)$ and $g(\nu) \preceq f(\nu)$.

Furthermore, we will need some classical notions from the theory of reversible Markov chains on finite state spaces. For every pair of states $x,y \in \cX$ for which $P(x,y) \neq 0$, we define the \textit{resistance} of the edge $e=(x,y)$ as $r(e)=r(x,y)=(\pi(x) \, P(x,y))^{-1}$. For any state $x \in \cX$ and subset $A \subset \cX \setminus \{x\}$, define then the \textit{effective resistance} $R(x \leftrightarrow A):= \pi(x)^{-1} \, \mathbb P (T_{x,A} < T_{x,x}^+)$ and the \textit{critical resistance} $\Psi(x,A):=\min_{\o: x \to A} \max_{ e \in \o} r(e)$. As shown~\cite[Proposition 2.2]{Hollander2017a}, effective resistance and critical resistance of the same pair of states are intimately related, namely there exists a constant $k \geq 1 $ independent of $\nu$ such that
\begin{equation}
\label{eq:criticaleffective}
	\frac{1}{k} \Psi(x,A) \leq R(x \leftrightarrow A) \leq k \, \Psi(x,A).
\end{equation}
As illustrated in~\cite[Section 2.1]{Hollander2017a}, the following identity holds for the expected hitting time of the target subset $A$ starting from a state $x \not\in A$
\begin{equation}
\label{eq:txa}
	\E T_{x,A} = \pi(x) \, R(x \leftrightarrow A) \sum_{z \in \cX} \frac{\pi(z)}{\pi(x)} \, \mathbb P (T_{z,x} < T_{z,A}),
\end{equation}
Consider now the special case relevant for this theorem in which the starting state is a dominant state, \ie $x=s \in \cD$, and the target subset $D$ is such that $D \subseteq \cD \setminus \{x\}$. From~\eqref{eq:statdist_multichannel_homogeneous} it immediately follows that
\[
	\pi(s) \succ \pi(z) \quad \forall \, z \in \cX \setminus \cD \, \text{ and } \, \pi(s) \succeq \pi(z) \quad \forall \, z \in \cD.
\]
Using this fact and the identity~\eqref{eq:txa} we deduce that $\E T_{s,D} \asymp \pi(s) \, R(s \leftrightarrow D)$ as $\ninf$. In view of~\eqref{eq:criticaleffective}, it then follows that
\[
	\E T_{s,D} \asymp \pi(s) \, \Psi(s,D), \quad \text{as } \ninf,
\]
and thus
\begin{equation}
\label{eq:ETPsi}
	\limnu \log_\nu \E T_{s,D} = \limnu \log_\nu \pi(s) \, \Psi(s,D).
\end{equation}
Consider the right-hand side of the latter identity. Using the definition of critical resistance, it follows that
\begin{align}
\log_\nu \pi(s)  \Psi(s,D) &= \log_\nu \min_{\o: s \to D} \max_{e \in \o} \pi(s) r(e)\nonumber\\
	&= \log_\nu \min_{\o: s \to D} \max_{(x,y) \in \o} \frac{\pi(s)}{\pi(x) \, P(x,y)}\nonumber\\
	&= \min_{\o: s \to D} \max_{(x,y) \in \o} \log_\nu \frac{C N \nu^{\ac}}{\nu^{a(x)} \nu^{-[a(x)-a(y)]^+}}\nonumber\\
	& = \log_\nu C N + \min_{\o: s \to D} \max_{(x,y) \in \o} \ac-a(x)+[a(x)-a(y)]^+\nonumber\\
	& = \log_\nu C N + \min_{\o: s \to D} \max_{(x,y) \in \o} \ac - ( a(x) \wedge a(y) )\nonumber\\
	& = \log_\nu C N + \min_{\o: s \to D} \max_{x \in \o} \ac-a(x)\nonumber\\
	& = \log_\nu C N + \D(s,D), \label{eq:chainid}
\end{align}
where the third last passage above follows from the identity $-a(x)+[a(x)-a(y)]^+ = - ( a(x) \wedge a(y) )$, while the second last one is a consequence of the fact that we take the maximum over \textit{all the configurations visited by the path} $\o$, rather than the edges crossed by the same.

Taking the limit in~\eqref{eq:chainid} we obtain $\limnu \log_\nu \pi(s)  \Psi(s,D) = \D(s,D)$, which, combined with identity~\eqref{eq:ETPsi}, yields
\[
	\limnu \log_\nu \E T_{s,D}(\nu) = \D(s,D).
\]
In view of~\eqref{eq:contdisc}, we have
\begin{align*}
	\limnu \log_\nu \E \tau_{s,D}(\nu)
	= \limnu \log_\nu (CN\nu)^{-1} \E T_{s,D}(\nu)
	= \limnu \log_\nu \E T_{s,D}(\nu) +  \limnu \log_\nu (CN\nu)^{-1} 
	= \D(s,D) -1,
\end{align*}
and this concludes the proof.\end{proof}


\section{Mixing time asymptotics for multi-channel CSMA networks}
\label{sec6}
In this section we turn attention to the long-run behavior of the Markov process $\{X_C(t)\}_{t \geq 0}$ and in particular examine the rate of convergence to the stationary distribution. We measure the rate of convergence in terms of the total variation distance and the so-called \textit{mixing time}, which describes the time required for the distance to stationarity to become small.

The mixing time becomes particularly relevant when the network has two or more dominant states which together attract the entire probability mass in the limit as $\ninf$. Indeed, in this case, the mixing time provides an indication of how long it takes the activity process to reach a certain level of fairness among the dominant components. 

For any fixed activation rate $\nu$, the maximal distance over $x\in\cX_C$, measured in terms of total variation, between the distribution at time~$t$ of the activity process and the stationary distribution $\pi_C$ is defined as
\[
	d_t(\nu):= \max_{x\in\cX_C} \| P^t_\nu (x, \cdot) -\pi_C\|_{\mathrm{TV}},
\]
where $P^t_\nu(x,\cdot)$ is the distribution at time $t$ of the Markov process $\{X_C(t)\}_{t \geq 0}$ started at time $0$ in $x$. The \textit{mixing time} of the process $\{X_C(t)\}_{t \geq 0}$ is then defined as
\[
	\tm(\e,\nu):=\inf\{t \geq 0 : d_t(\nu) \leq \e\}.
\]

The next result shows that the worst communication height between any two dominant states provides a lower bound for the timescale at which the activity process mixes.

\begin{theorem}[Mixing time timescale]\label{thm:mix}
The mixing time of the multi-channel activity process $\{X_C(t)\}_{t \geq 0}$ satisfies
\begin{equation}
\label{eq:mixbound}
	\limnu \log_\nu \tm(\e,\nu) \geq \GC -1,
\end{equation}
where $\GC \geq 1$ is the worst communication height between a pair of dominant states, \ie $\GC:=\max_{s,s' \in \cD} \Delta(s,s')$.
\end{theorem}
The proof of this theorem is based on a conductance argument that leverages the bottlenecks of the state space $\cX_C$ and is presented later in Subsection~\ref{sub:proofmix}.

The first consideration is that for some networks the mixing time worsens when the number of available channels increases. Consider the network depicted in Figures~\ref{fig:mixingincreases} and~\ref{fig:strict}, which illustrates its dominant states for $C=1$ and $C=2$, respectively. For such a network it is easy to check that~\eqref{eq:mixbound} holds with equality when $C=1$ and to compute that $\G(1)=2$, while $\G(2)=3$, see Table~\ref{tab:tabmix}.

\begin{figure}[!h]
\centering
\includegraphics[scale=1.15]{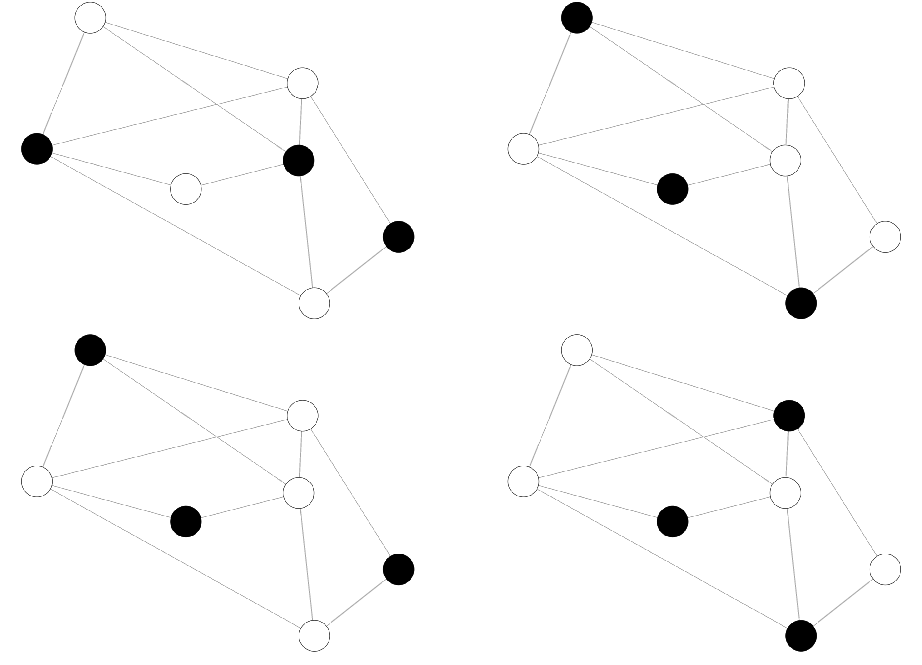}
\caption{The four dominant states $\mathcal{D}_1=\{s^{(1)},s^{(2)},s^{(3)},s^{(4)}\}$ on a network with $C=1$ channels available.}
\label{fig:mixingincreases}
\end{figure}

\begin{figure}[!h]
\centering
\includegraphics[scale=1.15]{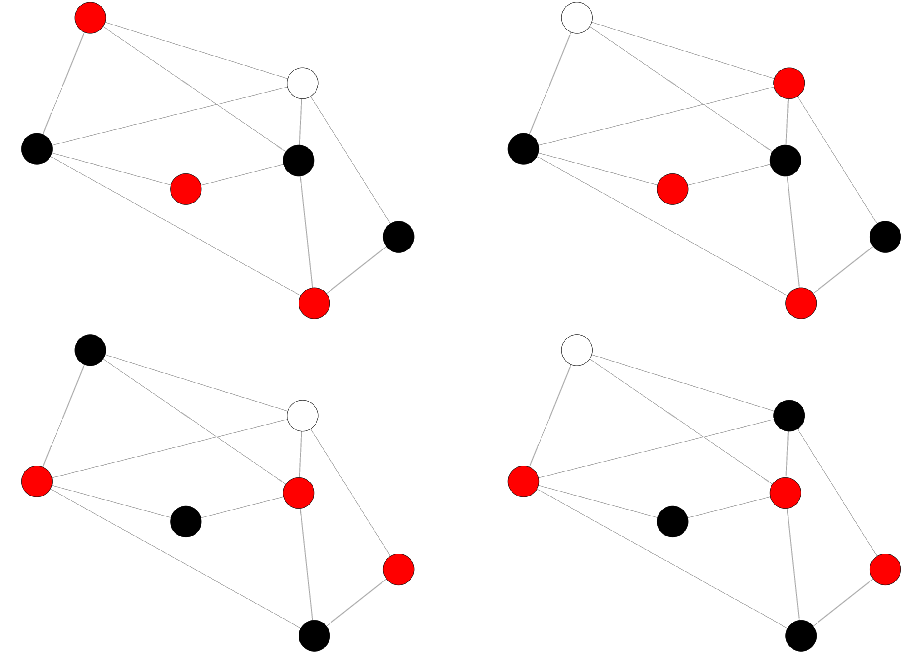}
\caption{The four dominant states $\mathcal{D}_2=\{s^{(5)},s^{(6)},s^{(7)},s^{(8)}\}$ on a network with $C=2$ channels available.}
\label{fig:strict}
\end{figure}

\begin{table}[!h]
\centering
\begin{tabular}{c|cccc}
& $s^{(1)}$ & $s^{(2)}$ & $s^{(3)}$ & $s^{(4)}$\\
\hline
$s^{(1)}$ & $0$ & $2$ & $2$ & $2$\\
$s^{(2)}$ & $2$ & $0$ & $1$ & $1$ \\
$s^{(3)}$ & $2$ & $1$& $0$ & $1$\\
$s^{(4)}$ & $2$ & $1$ & $1$ & $0$\\
\end{tabular}
\quad 
\begin{tabular}{c|cccc}
& $s^{(5)}$ & $s^{(6)}$ & $s^{(7)}$ & $s^{(8)}$\\
\hline
$s^{(5)}$ & $0$ & $1$ & $3$ & $3$\\
$s^{(6)}$ & $1$ & $0$ & $3$ & $3$ \\
$s^{(7)}$ & $3$ & $3$& $0$ & $1$\\
$s^{(8)}$ & $3$ & $3$ & $1$ & $0$\\
\end{tabular}
\vspace{0.5cm}
\caption{The values of the communication height $\Delta(\cdot,\cdot)$ between all pairs of dominant states in $\mathcal{D}_1$ (on the left, cf.~Figure~\ref{fig:mixingincreases}) and in $\mathcal{D}_2$ (on the right, cf.~Figure~\ref{fig:strict}).}
\label{tab:tabmix}
\end{table}
\FloatBarrier

We now compare this new exponent $\GC$ characterizing the mixing time in the regime $\ninf$ with the starvation index $\si$ introduced in Section~\ref{sec5}. The inequality $\siv \leq \GC$ holds for every node $i \in V$ with $\cS_C(i) \neq \emptyset$ and $\cS_C(i) \neq \cD$, since
\[
	\siv = \max_{s \in \cD \setminus \cS_C(i)} \min_{s' \in \cS_C(i)} \D(s,s') \leq \max_{s,s' \in \cD} \D(s,s') = \GC,
\]
and trivially holds also for all the other nodes for which we set $\siv=0$, as $\GC \geq 1$. Therefore, the starvation index of the network $\si$ is also not larger than $\GC$, \ie
\begin{equation}
\label{eq:indices}
	\si \leq \GC.
\end{equation}
These two indices can be equal, for instance in the case of a conflict graph with a grid structure and an even number of nodes. As illustrated in~\cite{ZBvLN13}, the CSMA dynamics with $C=1$ on this bipartite conflict graph has exactly two dominant states and each node is active in only one of the two, and therefore $\Upsilon(1)= \G(1)$.

However, for most conflict graphs and choices of the number $C$ of available channels inequality~\eqref{eq:indices} is strict. The network presented in Figure~\ref{fig:strict} is an example for which $\Upsilon(2) < \G(2)$. We already calculated earlier that $\G(2)=3$. As far as the starvation index is concerned, the only two nodes $v,w$ for which it is well defined are the two topmost ones of each dominant state in Figure~\ref{fig:strict} and using again Table~\ref{tab:tabmix} we can deduce that $\Upsilon_v(2)=\Upsilon_w(2)=1$ and thus $\Upsilon(2)=1$. 


The fact that inequality~\eqref{eq:indices} is often strict suggests that the mixing time may not be the best metric to study temporal starvation at least in the high-activation regime $\ninf$ since it overestimates the duration of the starvation effects. Indeed, while it may take a long time for the activity process $\{X_C(t)\}_{t \geq 0}$ to visit all the dominant states and thus ``mix'' on the state space $\cX_C$, the temporal starvation of the nodes often occurs at shorter timescales even in the worst case, that is the one captured by the network starvation index $\si$.

We notice that determining precisely for which class of conflict graphs the strict inequality $\si < \GC$ holds is probably a very difficult combinatorial problem, whose solution is beyond the scope of this paper.

\newpage
\subsection{Asymptotic lower bound for the mixing time}
\label{sub:proofmix}
This subsection is entirely devoted to the proof of Theorem~\ref{thm:mix}. The proof relies on a classical conductance argument, which is often used in the literature to obtain lower bounds for the mixing time of Markov chains with a finite state space, see~\cite[Theorem 7.3]{LPW09}.

For any subset $S \subset \cX_C$, the \textit{conductance of} $S$ is defined as
\[
	\Phi(S):=\frac{Q(S,S^c)}{\pi_C(S)},
\]
where $Q(S,S^c)$ is \textit{probability flow out of} $S$, \ie
\[
	Q(S,S^c):=\sum_{x \in S, y \in S^c} \pi_C(x) \, q(x,y).
\]
In order to bound the mixing time of the activity process $\{X_C(t)\}_{t\geq 0}$, we will make use of the continuous-time counterpart of that theorem, which is summarized in the following lemma, see~\cite[Lemma 7.1]{ZBvL15} for a proof. 
\begin{lemma}{} \label{lem:cond}
For any $\e \in (0,1/2)$,
\begin{equation}
\label{eq:mixgen}
	\tm(\e, \nu) \geq \frac{1/2 -\e}{ \min_{S \subset \cX_C \,: \, \pi_C(S) \leq 1/2} \Phi (S)}.
\end{equation}
\end{lemma}
To get a sharp bound for the conductance and hence a sharp lower bound for the mixing time, we need to identify a subset $S$ with low conductance.
Let $s,s' \in \cD$ be two dominant states such that $\Delta(s,s')=\GC$ and consider the subset $S$ of network states that can be reached by the CSMA dynamics starting in $s$ while keeping no less that $\ac-\GC+1$ nodes active, \ie
\[
	S=\{ x \in \cX_C : \Delta(s,x) < \GC\}.
\]
Note that $s \in S$ and $s' \in S^c$, which implies that both subsets $S$ and $S^c$ have non-vanishing stationary probabilities containing at least one dominant set each. Without loss of generality, let us assume that $\pi_C(S) \leq 1/2$ for all $\nu$ sufficiently large.

By construction of $S$, all states $x \in S$ such that there exists $y \in S^c$ such that $d(x,y) =1$ must have exactly $\ac-\GC+1$ active nodes. Indeed, they clearly cannot have fewer active nodes, otherwise $\Delta(s,x) \geq \GC$, and if they had strictly more than $\ac-\GC+1$, it would be impossible to exit in a single step from $S$, since the CSMA dynamics prescribes the number of active nodes to increase or decrease at most by one. In view of this fact, the probability flow of $S$ can be rewritten as
\[
	Q(S,S^c) = \sum_{x \in S, y \in S^c} \pi_C(x) \, q(x,y) = Z^{-1} \nu^{\ac-\GC+1} B(S),
\]
where the quantity $B(S):=|\{ (x,y) \in S \times S^c  \,:\, d(x,y) =1\}|$ counts the number of possible way of exiting from $S$, which is independent of $\nu$ and depends only on the structure of the state space $\cX_C$. Then,
\begin{align*}
	\limnu \log_\nu Q(S,S^c)
	& = \limnu \log_\nu (Z^{-1} \nu^{\ac-\GC+1} B(S))\\
	& = \limnu \left (\log_\nu (Z^{-1} \nu^{\ac-\GC+1}) + \log_\nu B(S) \right ) \\
	& = \limnu \log_\nu (Z^{-1} \nu^{\ac-\GC+1})\\
	& = \ac-\GC+1 - \ac = 1-\GC.
\end{align*}
Using the conductance of the subset $S$ in~\eqref{eq:mixgen} yields the lower bound $\tm(\e, \nu) \geq (1/2 -\e) / \Phi (S)$. Therefore,
\begin{align*}
	\limnu \log_\nu \tm(\e, \nu) 
	\geq \limnu \log_\nu \frac{1/2 -\e}{\Phi (S)} 
	 = - \limnu \log_\nu \Phi (S)
	=  \limnu \log_\nu \pi (S) - \log_\nu Q(S,S^c)
	 = \GC-1,
\end{align*}
and this concludes the proof of Theorem~\ref{thm:mix}.

\section{Model generalizations}
\label{sec7}
In this section, we discuss some generalizations of our model and show to which extent our results remain valid.

Our framework for transition and mixing times does allow for more general heterogeneous activation rates of the form $\nu_{i,c} = f_{i,c}(\nu)$, where $f_{i,c}(\cdot)$ are non-negative real functions with a common scaling parameter $\nu$. These more general rates can be used to model interesting scenarios where nodes have preferred channels or frequency channels with different features. Allowing for this more general parametric family of activation rates does not affect the state space $\cX_C$, but does change the stationary distribution $\pi_C$ and, ultimately, the dominant states of the process. The asymptotic results derived in Sections~\ref{sec5} and~\ref{sec6} in the regime $\ninf$ still hold for this more general setting. However, it is not enough to keep track only of the number of active nodes and one needs to work with the more involved functions $\widetilde{a}(x):= \sum_{i=1}^N \sum_{c=1}^C w_{i,c} \mathds{1}_{\{x_i=c\}}$ and $\widetilde{\mathcal{A}}(C):=\max_{x \in \cX} \widetilde{a}(x)$, where the weights $w_{i,c}$ are defined as $w_{i,c}:=\limnu \log_\nu f_{i,c}(\nu)$ (assuming that such limits exist and are finite). 

Another natural generalization of our model is the one where each node $i \in V$ of the network possibly has more than one radio interface and thus can simultaneously transmit on $c_i$ channels, with $1 \leq c_i \leq C$, having an independent back-off timer for each one of them. In this setting there is no natural way to represent the multi-channel dynamics as a single-channel dynamics on a virtual network as we did in Section~\ref{sec3}, but the resulting Markov process still has a product-form stationary distribution, as proved in~\cite{PYLC10}. Despite the fact that this assumption results in a much larger and fundamentally different state space, the asymptotic results for transition and mixing times proved in Sections~\ref{sec5} and~\ref{sec6} remain valid, after having opportunely modified the definition of the functions $a(x)$ and $\mathcal{A}(C)$. However, the identification of the dominant states becomes more involved, especially in the case where a node cannot be active with different radio interfaces on the same frequency channel.

As far as the results of Section~\ref{sec4} are concerned, none of the statements of Theorem~\ref{thm:throughput} but the first one hold in general for the two aforementioned model extensions or in the case where the interference on different channels is described by different conflict graphs. Indeed its proof used in a crucial way the equivalent description of a network activity state as an (unweighted) partial $C$-colorings of the conflict graph, which is not valid anymore in the more general settings described above. 

The only exception is the scenario in which every node has exactly $C$ radio interfaces, each one dedicated to precisely one of the $C$ frequency channels. The dynamics on different channels become in this way decoupled and the network dynamics can be described by means of $C$-dimensional Markov process, each component being an independent single-channel CSMA Markov process. It is easy to show that in this case the aggregate throughput $\Theta(C)$ is constant in $C$, meaning that increasing the number $C$ of available channels may alleviate the starvation effects without negatively affecting the throughput.

\section{Conclusions}
\label{sec8}
In this paper we investigate the performance of random-access networks in which $\nc$ non-overlapping channels are available for transmission. We consider a Markovian model for the multi-channel network dynamics evolving according to a CSMA-like algorithm, aiming to understand how the network throughput performance and starvation issues depend on the number of available channels.

The most relevant scenario for our analysis is the one where the number $\nc$ of available channels is not sufficient to allow all nodes to be active simultaneously and for this reason temporal starvation phenomena persist. Focusing on the high-activation regime, we show how the activity states with a maximum number of active nodes play a crucial role in determining both the aggregate throughput and the temporal starvation effects in the network.

We then characterize the asymptotic behavior of transition times between dominant states by leveraging specific structural properties of the state space where the activity process evolves. This analysis allows us to infer the precise timescales at which starvation phenomena occur in the network, paving the way for a detailed analysis of the delay performance of multi-channel CSMA networks.

The same analytical framework that we develop for expected hitting times also yields an asymptotic lower bound for the mixing time of the process, which we show tends to overestimate the magnitude of the temporal starvation phenomena.

%
%
\newpage
\bibliographystyle{abbrv}

\end{document}